\documentclass[psamsfonts,12pt]{amsart}

\usepackage{enumerate}
\usepackage{amsmath}
\usepackage{color}

\newcommand{\R}{\mathbb{R}}
\newcommand{\N}{\mathbb{N}}

\newcommand{\di}{\mathop{\mbox{\rm{diad}}}\nolimits}
\newcommand{\cD}{\mathcal{D}}

\newtheorem{thm}{Theorem}
\newtheorem*{thm*}{Theorem}

\newtheorem{cor}[thm]{Corollary}
\newtheorem{prop}[thm]{Proposition}
\newtheorem{obv}{Observation}
\newtheorem{que}{Question}
\newtheorem{lem}[thm]{Lemma}
\theoremstyle{definition}

\newtheorem{rem}{Remark}

\begin{document}


\title[Characterization of quasi-arithmetic means\dots]{Characterization of quasi-arithmetic means without regularity condition}
\author{P\'al Burai, Gergely Kiss and Patricia Szokol}
\address{P\'al Burai, \newline University of Debrecen, \newline 4028, Debrecen, 26 Kassai road, HUNGARY}
\email{burai.pal@inf.unideb.hu}
\address{Gergely Kiss, \newline Alfr\'ed R\'enyi Institute of Mathematics,\newline Budapest, Re\'altanoda street 13-15, 1053 HUNGARY}
\email{kiss.gergely@renyi.hu}
\address{Patricia Szokol, \newline University of Debrecen,\newline
MTA-DE Research Group “Equations, Functions and Curves”, \newline
4028, Debrecen, 26 Kassai road, Hungary}
\email{szokol.patricia@inf.unideb.hu}
\maketitle


\begin{abstract}
In this paper we show that bisymmetry, which is an algebraic property, has a regularity improving feature. More precisely, we prove that every bisymmetric, partially strictly monotonic, reflexive and symmetric function $F:I^2\to I$ is continuous. As a consequence, we obtain  a finer characterization of quasi-arithmetic means than the classical results of Acz\'el \cite{Aczel1948}, Kolmogoroff \cite{Kolmogoroff1930}, Nagumo \cite{Nagumo1930} and de Finetti \cite{deFinetti1931}.
\end{abstract}


\section{Introduction}

Our main goal is to prove a somewhat surprising result, namely, a purely algebraic property (bisymmetry) has a regularity improving feature. In other word, a better analytic behaviour of a map is resulted by a non-analytic object.

The story of bisymmetry is an old one. It is developed, together with its role in the characterization of quasi-arithmetic means, by J\'anos Acz\'el in 1948 (see \cite{Aczel1948}).

Quasi-arithmetic means are central objects in theory of  functional equations, in particular in the theory of means (see e.g. \cite{Burai2013c}, \cite{Daroczy2013}, \cite{Duc2020},  \cite{Gaal2018}, \cite{Glazowska2020}, \cite{Kiss2018}, \cite{Nagy2019},  \cite{Pales2011},  \cite{Pales2020},  \cite{Pasteczka2018},  \cite{Pasteczka2020}).

Acz\'el's motivation were the works of Kolmogoroff \cite{Kolmogoroff1930}, Nagumo \cite{Nagumo1930} and de Finetti \cite{deFinetti1931}. They considered quasi-arithmetic means as a sequence of maps, where the $n$th member of this sequence is the $n$-variable quasi-arithmetic mean. They characterize this sequence by means of reflexivity, continuity, increasing property, symmetry and associativity in the sense of Kolmogoroff\footnote{Associativity in the sense of Kolmogoroff (alternative names associativity with repetitions or decomposability see \cite{Grabisch2009}) is different from the classical associativity. A sequence of functions $\{F_n\}_{n\in\N}$, where $F_n\colon I^n\to\R$ is an  $n$-place map, is associative in the sense of Kolmogoroff if $F_n(x_1,\dots,x_k,x_{k+1},\dots,x_n)=F_n(F_k(x_1,\dots,x_k),\dots,F_k(x_1,\dots,x_k),x_{k+1},\dots,x_n)$ for every $k\in\{1,\dots,n\},\ n\in\N$ and $x_1,\dots,x_n\in I$.}.

Acz\'el dealt only with the two-variable case, and he changed associativity in the sense of Kolmogoroff into bisymmetry.
 In his proof (see \cite{Aczel1948} or \cite[proof of Theorem 1 on page 290]{Aczel1989}) continuity is used essentially but a little bit furtively. This was our motivation for a sophisticated examination of the proof in question.
It turned out at last that continuity is superfluous in the characterization of two-variable quasi-arithmetic means, which is somehow a striking outcome of the present investigation.
In the second section we will see that the situation is completely different if we assume associativity, which is a close relative of bisymmetry.

Concerning the structure of our work, we summarize the related concepts and results in the next section. The third one is devoted to our  main theorem and its consequences together with their proofs. In the last section we pose some open questions connecting to our current work.

\section{Preliminary results and notations}\label{S:preliminary}

We keep the following notations throughout the text. Let $I\subseteq\R$ be a  proper interval (i.e. the cardinality of $I$ is at least $2$, in notation $|I|\geq 2$) and $F\colon I\times I\to\R$ be a map.

The map $F$ is said to be
\begin{enumerate}[(i)]
\item {\it reflexive}, if $F(x,x)=x$ for every $x\in I$;
\item {\it partially strictly monotone / partially strictly increasing / partially monotone / partially increasing}, if the functions $x\mapsto F(x,y_0)$, $y\mapsto F(x_0,y)$ are strictly monotone / strictly increasing / monotone / increasing for every fixed $x_0\in I$ and $y_0\in I$, respectively;
\item {\it symmetric}, if $F(x,y)=F(y,x)$ for every $x,y\in I$;
\item {\it bisymmetric}, if
\begin{equation}\label{E:bisymmetry}
\tag{BS} F\big(F(x,y),F(u,v)\big)=F\big(F(x,u),F(y,v)\big)
\end{equation}
for every $x,y,u,v\in I$;
\item  {\it associative}, if
\begin{equation}\label{E:assoc}
\tag{AS} F\big(F(x,y),z)\big)=F\big(x,F(y,z)\big)
\end{equation}
for every $x,y,z\in I$.
\end{enumerate}
The map $F$ is said to be a \textit{mean} on the interval $I$ if its value is always between the minimum and the maximum of the variables, that is to say,
\begin{equation}
\min\{x,y\}\leq F(x,y)\leq\max\{x,y\}
\end{equation}
for every $x,y\in I$. If the previous inequalities are strict, whenever $x$ is different from $y$, then $F$ is called a \textit{strict mean} on $I$.

\begin{obv}\label{Obv:F_is_a_mean}
If a map $F$ is reflexive and partially strictly increasing, then it is a strict mean on $I$.
\end{obv}
\begin{proof}
Let $x,y\in I$ be arbitrary. We can assume without the loss of generality that $x<y$. Then, using the assumptions, we can write
\[
x=F(x,x)<F(x,y)<F(y,y)=y.
\]
\end{proof}

The following fundamental result is due to Acz\'el \cite{Aczel1948}, which also can be found in \cite[Theorem 1. page 287]{Aczel1989}.

\begin{thm}\label{Aczel1}
A function $F:I^2\to I$ is continuous, reflexive, partially strictly monotonic, symmetric and bisymmetric  mapping if and only if there is a continuous, strictly increasing function $f\colon [0,1]\to I$ that satisfies
\begin{equation}\label{eqa1}
    F(x,y)=f \left(\frac{f^{-1}(x)+f^{-1}(y)}{2}\right),\qquad x,y\in I.
\end{equation}
\end{thm}
The following observation indicates that monotonicity of F in Theorem \ref{Aczel1}  can be substituted by
increasing property. Hence, in the
sequel we would focus on partially increasing mappings.
\begin{obv}
Let $F: I^2\to I$ be a reflexive, partially strictly monotone, symmetric mapping. Then $F$ is partially strictly increasing.
\end{obv}
\begin{proof}
Clearly, if $F$ is partially strictly monotone and symmetric on the interval $I$, then it is either strictly increasing in each of its variables or strictly  decreasing in each of its variables. Reflexivity implies that $F$ can only be increasing in each of its variables.
\end{proof}

The function $F$ of the form \eqref{eqa1} is called {\it quasi-arithmetic mean} as it is a conjugate of the arithmetic mean by an order preserving bijection $f$.

Apart from quasi-arithmetic means, bisymmetry is strongly connected to associativity.
 The following theorem due to Acz\'el is well-known (see e.g. \cite[Theorem 1. page 107]{Aczel1987}).

\begin{thm}\label{Aczel2}
A function $F:I^2\to I$ is an associative, partially strictly increasing and continuous mapping if and only if there is a proper interval $J\subset\R$ and a continuous and strictly monotonic function $f\colon J \to I$  that satisfies
\begin{equation}\label{eq0.1}
    F(x,y)=f \left(f^{-1}(x)+f^{-1}(y)\right),\qquad x,y\in I.
\end{equation}
\end{thm}

It is important to note that the original theorem assumes cancellative\footnote{A map $F\colon I^2\to I$ is said to be \emph{cancellative} if either $F(x,a)=F(y,a)$ or $F(a,x)=F(a,y)$ implies $x=y$ for every $x,y,a\in I$.} property instead of partially strictly increasing property, however, every partially strictly increasing two-place operation is automatically cancellative, and every continuous cancellative operation on an interval is partially strictly monotone. A simpler and constructive proof of Theorem \ref{Aczel2}, given by Craigen and P\'ales, can be found in \cite{Craigen1989}. The $n$-variable case proved by Couceiro and Marichal in \cite{Couceiro2012}.

It is clear that every map is of the form \eqref{eq0.1} is bisymmetric, as well.

This theorem immediately implies the following result.
\begin{cor}
There is no reflexive, associative, partially strictly increasing, symmetric and continuous mapping on $I^2$, where $I$ is an arbitrary interval.
\end{cor}
\begin{proof}
According to Theorem \ref{Aczel2} the assumptions without reflexivity implies that the map in question can be written in the form of \eqref{eq0.1}, which is clearly not reflexive.
\end{proof}

We show similar implication in more general settings, where the concepts (reflexivity, associativity, etc.) on a totally ordered set $X$ can be defined in a similar way as on an interval $I$.
\begin{prop}\label{prop1}
Let $X$ be a strictly totally ordered set with $|X|\ge 2$. Then there is no reflexive, associative, partially strictly increasing mapping on $X^2$.
\end{prop}

\begin{proof}
Let us assume that $F\colon X^2\to X$ is  reflexive, associative and partially strictly increasing map. Then for arbitrary $a,b\in X,\ a<b$ we have $a<F(a,b)<b$. Using associativity and reflexivity we can write
\[
F(a,F(a,b))=F(F(a,a),b)=F(a,b),
\]
which contradicts to the partially strictly increasing property of $F$.
\end{proof}

On the other hand, it is well-known that if we weaken the conditions such as the partial functions are increasing (not strictly), then uncountably many such functions do exist even if they are reflexive.

 The general description of reflexive, associative, partially increasing functions are not known. However, assuming the existence of a neutral element, we have the following characterization.
\begin{thm}\label{thmCD}\cite[Theorem 2.2.]{Kiss2018a}
A function $F:I^2\to I$ is reflexive, associative, partially increasing, and has a neutral element  $e\in I$ (i.e. $F(x,e)=F(e,x)=x$ for every $x \in I$) if and only if there exists a monotone decreasing function $g :
I \to I$ with $g(e) = e$ such
that
\begin{equation}\label{eqas}
F(x, y) =\begin{cases}
\min (x, y),& \textrm{if } y < g(x), \textrm{ or } \\
& y=g(x)\textrm{ and }x<g^2(x)\\
\max (x, y),& \textrm{if } y > g(x), \textrm{or }\\
& y=g(x)\textrm{ and }x>g^2(x),\\
\min (x, y)\textrm{ or }\max (x, y),& \textrm{if } y = g(x) \textrm{and } x=g^2(x).
\end{cases}
\end{equation}
Moreover, $F(x,y)=F(y,x)$ except perhaps the set of points $(x,y)\in I^2$ satisfying $y=g(x)$ and $x=g^2(y)$.
\end{thm}

\begin{rem}
Clearly, such an $F$ in Theorem \ref{thmCD} is symmetric if $F(x,y)=F(y,x)$ when $g(x)=y$, and it is
continuous, if $F$ is the minimum or the maximum on $I^2$.
\end{rem}

The following lemma is a folklore, the interested reader is referred to \cite[Lemma 22.]{Couceiro2018}.
\begin{lem}
Let $X$ be a set with cardinality $|X|\ge 2$ and let $F: X^2 \to X$ be a map. If F is bisymmetric and has a neutral element, then it is associative and symmetric.
\end{lem}
Thus we get the following result as a corollary.
\begin{cor}
Let $F:I^2\to I$ be a reflexive, bisymmetric, partially increasing function that has a neutral element $e\in I$. Then $F$  is symmetric, satisfies \eqref{eqas} with $F(x,y)=F(y,x)$ if $g(x)=y$.

If $F$ is continuous, then $F$ is the minimum or the maximum on $I^2$.
\end{cor}

Thus, for such a bisymmetric family of functions the continuity assumption is essential.  One can find uncountably many different discontinuous functions satisfying \eqref{eqas}.

It is also worthy to note that the projections $F_1(x,y)=x$  and $F_2(x,y)=y$ are reflexive, bisymmetric, partially increasing functions and continuous but not symmetric (and have no neutral element).

\section{Main results}

\begin{thm}\label{T:bisymmetryimpliescontinuity}
Let us assume that $a,b\in\R,\ a<b$ and $F\colon [a,b]^2\to[a,b]$ is a reflexive, partially strictly increasing, symmetric and bisymmetric map. Then there is a continuous function $f\colon [0,1]\to [a,b]$ such that
\begin{equation}\label{Eq_foalak}
F(x,y)=f\left(\frac{f^{-1}(x)+f^{-1}(y)}{2}\right),\qquad x,y\in[a,b].
\end{equation}

\end{thm}
\begin{proof}
At first we imitate Acz\'el's algorithmic construction of function $f$ (see  Acz\'el and Dhombres \cite{Aczel1989} on the pages $287-290$).
Thus let $f\colon[0,1]\to[a,b]$ be a function, defined recursively on the set of dyadic numbers $\di[0,1]$. We introduce the following notations for the sake of simplicity:
\[
\cD:=\di[0,1],\qquad x\circ y:=F(x,y),\quad x,y\in[a,b].
\]Let
\begin{align*}
    f(0)=a, \ &  \ f(1)=b\\
    f\left(\frac{1}{2}\right)=a\circ b, \ \ f\left(\frac{1}{4}\right)=a\circ(&a\circ b),  \ \ f\left(\frac{3}{4}\right)=a\circ(b\circ b)
    \end{align*}
     and so forth. The function $f$ is defined by satisfying the following identity
\begin{equation}\label{identity_on_dyadics}
f\left(\frac{d_1+d_2}{2}\right)=f(d_1)\circ f(d_2)
\end{equation}
for every $d_1,d_2\in\cD$.
In \cite{Aczel1989} it was shown that the function $f$ is well-defined and strictly increasing on $\cD$.\footnote{We note that there was also shown in \cite{Aczel1989} that if $F$ is continuous, then  so is $f$. We do not assume continuity here.}

In the sequel we show that $f(\cD)$ is dense in $[a,b]$, which implies that $f$ is continuous.

Suppose for the contrary that $f(\cD)$ is not dense in $[a,b]$. In the following steps we show that this is impossible.

\textbf{First step:} We prove that $\overline{f(\cD)}$ (the closure of $f(\cD)$) is uncountable. For this, we construct an extension of $f$ on the whole unit interval.

Let $x\in[0,1]\setminus\cD$. Then there is a strictly monotone increasing sequence $\{d_n\}_{n\in\N}$ in $\cD$ such that $d_n\to x$ as $n\to\infty$. Because $f$ is strictly increasing on $\cD$ (see \cite{Aczel1989}) and $[a,b]$ is compact, we have that $f(d_n)$ is convergent. Let us define $f(x)$ as the limit of this sequence, that is
\[
f(x):=\lim\limits_{n\to\infty}f(d_n).
\]
If $y\in[0,1]\setminus\cD$ and $x\not= y$, then we can assume that $x<y$. Then, there exists a strictly monotone increasing dyadic sequence $\{s_n\}_{n\in\N}$ such that $s_n\to y$. If $j\in\N$ is large enough, we have that
\[
d_i<s_j, \qquad i\in\N,
\]
which entails
\[
f(d_i)<f(s_j)<f(y).
\]
Hence,
\[
f(x)=\lim\limits_{n\to\infty}f(d_n)\leq f(s_j)<f(y).
\]
In other words, as $f$ is strictly increasing on $\cD$, we have that its extension is strictly increasing on the whole unit interval. So, it is injective, which implies that $\overline{f(\cD)}$ is uncountable.

\textbf{Second step:} We prove that $\overline{f(\cD)}$ has uncountably many two-sided accumulation points, which can be defined as follows.

Let $H\subset \R$ be a set. A point $\alpha$ of $H$ is said to be
\begin{itemize}
\item \textit{isolated}, if there exists an $\varepsilon>0$ such that
\[
]\alpha-\varepsilon,\alpha+\varepsilon[~\cap~H =\emptyset,\]
\item \textit{two-sided accumulation point}, if for every
$\varepsilon>0$, we have
\[
]\alpha-\varepsilon,\alpha[~\cap~H\not=\emptyset\quad\mbox{ and }\quad ]\alpha,\alpha+\varepsilon[~\cap~H\not=\emptyset.
\]

\item \textit{one-sided accumulation point}, if it is neither isolated, nor two-sided accumulation point.
\end{itemize}

The set $\overline{f(\cD)}$ has at most countably many isolated points, otherwise there would be uncountably many disjoint, proper intervals in the compact interval $[a,b]$. From the same reason, $\overline{f(\cD)}$ has at most countably many half-sided accumulation points.

Since $\overline{f(\cD)}$ is uncountable, we have that it has at least uncountably many two-sided accumulation points.

\textbf{Third step:} Suppose that $f(\cD)$ is not dense in $[a,b]$. Then there is a point $z\in]0,1[$ such that
\begin{equation}\label{sequences_tending_to_z}
\lim\limits_{d_i\to z-}f(d_i)<\lim\limits_{D_i\to z+}f(D_i),
\end{equation}
where $\{d_i\}_{i\in\N}$ and $\{D_i\}_{i\in\N}$ are arbitrary dyadic sequences from $\cD$ tending to $z$ from the left and from the right, respectively. Let us define $X,Y\in[a,b]$ in the following way:
\[
X:=\lim\limits_{d_i\to z-}f(d_i),\qquad Y:=\lim\limits_{D_i\to z+}f(D_i).
\]
The values $X$ and $Y$ are independent from the choices of the sequences $\{d_i\}_{i\in\N}$ and $\{D_i\}_{i\in\N}$, respectively (see \cite{Aczel1989} on page 289).

\textbf{Fourth step:}
We prove that for arbitrary $\alpha\not=\beta$ two-sided accumulation points we have
\[
]\alpha\circ X,\alpha\circ Y[~\cap~ ]\beta\circ X,\beta\circ Y[ ~=\emptyset,
\]
where $X$ and $Y$ were defined in the previous step.

Let $\alpha,\beta\in[a,b],\ \alpha<\beta$ be two-sided accumulation points of $\overline{f(\cD)}$. Then, there are $d_\alpha,d_\beta\in\cD$ such that
\begin{equation}\label{E:alphabeta}
\alpha<f(d_\alpha)<f(d_\beta)<\beta.
\end{equation}
Moreover, there exist dyadic sequences $\{d_n\}_{n\in\N}$, and $\{D_n\}_{n\in\N}$ in the interval $[0,1]$ such that
\[
d_1<d_2<\cdots<z<\cdots<D_2<D_1,
\]
and
\[
\qquad f(d_n)\to X\qquad\mbox{and}\qquad f(D_n)\to Y
\]
as $n\to\infty$.
It also follows from the definitions of the sequences and from the strictly increasing property of $f$ on $\cD$ that
\begin{equation}\label{E:XY}
f(d_n)<X<Y<f(D_m)
\end{equation}
for every $n,m\in\N$.

If $n$ and $m$ are large enough, then
\[
\frac{D_m+d_\alpha}{2}<\frac{d_n+d_\beta}{2}.
\]
Applying \eqref{identity_on_dyadics} and the strictly increasing property of $f$ on $\cD$, we can write
\begin{eqnarray*}
f(d_n)\circ f(d_\alpha)<f(D_m)\circ f(d_\alpha)=f\left(\frac{D_m+d_\alpha}{2}\right)<\\
\\
f\left(\frac{d_n+d_\beta}{2}\right)=f(d_n)\circ f(d_\beta)<f(D_m)\circ f(d_\beta).
\end{eqnarray*}
It follows then
\[
]f(d_n)\circ f(d_\alpha),f(D_m)\circ f(d_\alpha)[~\cap~ ]f(d_n)\circ f(d_\beta),f(D_m)\circ f(d_\beta)[ ~=\emptyset.
\]
By the chain of inequalities \eqref{E:XY}, we have that
\begin{eqnarray*}
]X\circ f(d_\alpha),Y\circ f(d_\alpha)[ ~\subset~ ]f(d_n)\circ f(d_\alpha),f(D_m)\circ f(d_\alpha)[,&\\
]X\circ f(d_\beta),Y\circ f(d_\beta)[~ \subset ~]f(d_n)\circ f(d_\beta),f(D_m)\circ f(d_\beta)[,&
\end{eqnarray*}
which implies that
\[
]X\circ f(d_\alpha),Y\circ f(d_\alpha)[~\cap~ ]X\circ f(d_\beta),Y\circ  f(d_\beta)[ ~=\emptyset.
\]
Even more so, applying \eqref{E:alphabeta}, we have
\[
]X\circ \alpha,Y\circ \alpha[~\cap~ ]X\circ \beta,Y\circ \beta[ ~=\emptyset.
\]

It means, that $[a,b]$ contains uncountably many accumulation points and hence uncountably many disjoint, proper intervals, which is impossible. So, $f(\cD)$ is necessarily dense in $[a,b]$.

Now, we are going to prove, that the density of $f(\cD)$ implies the continuity of $f$. Let $x\in[0,1]\setminus\cD$ and let $\{d_n\}_{n\in\N},\ \{D_n\}_{n\in\N}$ be dyadic sequences from $\cD$ such that $d_n$ tends to $x$ strictly monotone increasingly from the left, and $D_n$ tends to $x$ strictly monotone decreasingly from the right. Let us consider the values
\[
L:=\lim\limits_{n\to\infty}f(d_n),\qquad R:=\lim\limits_{n\to\infty}f(D_n).
\]
Because of the strictly increasing property of $f$ it is clear that $L\leq R$.
If $L<R$, then
\[
f(\cD)~\cap ~]L,R[~=\emptyset,
\]
which contradicts to the fact that $f(\cD)$ is dense in $[a,b]$. Consequently, we obtain that $f$ is continuous and strictly monotone increasing on $[0,1]$, which entails that $f$ fulfils  \eqref{identity_on_dyadics} on $[0,1]$. Thus, we get the desired form \eqref{Eq_foalak} of $F$ on the whole interval $[a,b]$.
\end{proof}

\begin{thm}\label{fottetel}
Let $F:I^2\to I$ be a reflexive, partially strictly increasing, symmetric and bisymmetric mapping. Then $F$ is continuous.
\end{thm}

\begin{proof}
If $I$ is compact then we have the statement from Theorem \ref{T:bisymmetryimpliescontinuity}. Otherwise, one can approximate $I$ by a sequence of compact subintervals of $I$.\footnote{For further details see  \cite[the end of the proof of Theorem 1. on page 290-291]{Aczel1989}.
}
\end{proof}

As an immediate consequence of the previous theorem we get a more natural new characterization theorem of quasi-arithmetic means.
\begin{cor}
A function $F\colon I^2\to I$ is a quasi-arithmetic mean if and only if it is reflexive, partially strictly increasing, symmetric and bisymmetric.
\end{cor}

\begin{proof}
If $F$ is of the form  \eqref{eqa1}, then it is trivially reflexive, partially strictly increasing, symmetric and bisymmetric.

The opposite direction comes from Theorem \ref{fottetel}.
\end{proof}

\section{Further directions}

\subsection*{Problems connecting to the bivariate case}
The following Theorem can be found in \cite[p. 296]{Aczel1989}.
\begin{thm}\label{Aczel3}
Let $F:I^2\to I$ be a partially strictly monotonic and bisymmetric continuous mapping.  Then
\begin{enumerate}
    \item there are constants $A,B,C\in\R,\ AB\not=0$, and a continuous, strictly monotonic function $f\colon J\to I$ that satisfies
\begin{equation}\label{eqa1x2}
    F(x,y)=f \left(Af^{-1}(x)+Bf^{-1}(y)+C\right),\quad x,y\in I,
\end{equation}
    \item $F$ is reflexive if and only if there is a continuous, strictly monotonic $f\colon J\to I$ that satisfies
\begin{equation}\label{eqa1x3}
    F(x,y)=f \left(rf^{-1}(x)+(1-r)f^{-1}(y)\right),\quad x,y\in I, r\in\R\setminus\{0,1\},
\end{equation}
\end{enumerate}
where $J\subset\R$ is a proper interval.
\end{thm}

The proof of Theorem \ref{Aczel3} is based on the reflexive, symmetric case of Theorem \ref{Aczel1} and it relies on the fact that some functions are continuous such as $g_a(z)=F(F(a,z),F(z,a))$.  This naturally motivates the following open problems.

\begin{que}
 Is that true that every partially strictly monotonic and bisymmetric function $F:I^2\to I$ is automatically continuous, that satisfies \eqref{eqa1x2}?
\end{que}

\begin{que}
 Is that true that every reflexive, partially strictly monotonic and bisymmetric function $F:I^2\to I$ is automatically continuous, that satisfies \eqref{eqa1x3}?
\end{que}

\subsection*{Problems connecting to the $n$-arity case}

Analogously we can extend every property of $F$ introduced in Section \ref{S:preliminary} to $n$-ary functions. Hence we can talk about reflexivity, partially (strict) monotonicity, continuity of $G:I^n\to I$.

Symmetry is defined as $$G(x_1, \dots, x_n)=G(x_{\sigma(1)},\dots, x_{\sigma(n)})$$ holds for all $x_1,\dots, x_n$ and $\sigma\in S_n$, where $S_n$ is the permutation group of $n$ elements.

Bisymmety is defined as \begin{align*}
    &G(G(x_{1,1},\dots, x_{1,n}),G(x_{2,1},\dots, x_{2,n}),\dots, G(x_{n,1},\dots, x_{n,n}))=\\
    &G(G(x_{1,1},\dots, x_{n,1}),G(x_{1,2},\dots, x_{n,2}),\dots, G(x_{1,n},\dots, x_{n,n})),
\end{align*}
 for all $x_{i,j}\in[a,b]$. This property has a significant role in economics, especially, in the theory of aggregation functions (see e.g. \cite{Aczel1996}, \cite{Aczel2000}, \cite{Grabisch2009}, \cite{Maksa1999a}).

The following general questions arise naturally as $n$-ary analogue.
\begin{que}
Let $G:I^n\to I$ be a partially strictly increasing and bisymmetric function. Is that true that $G$ is continuous.
\end{que}
This problem seems too general at this moment.
Therefore we formalize the following direct analogue of our main result.
\begin{que}
Let $G:I^n\to I$ be a reflexive, partially strictly increasing, symmetric and bisymmetric function. Is that true that there is a proper interval $J\subset\R$ and a continuous function $f\colon J\to I$, such that
\[
G(x_1, \dots, x_n)=f\left(\frac{f^{-1}(x_1)+\dots +f^{-1}(x_n)}{n}\right),\qquad x_1, \dots,x_n\in I,
\] hence $G$ is continuous?
\end{que}
\section*{Acknowledgement}
The second author was supported by Premium Postdoctoral Fellowship of the Hungarian Academy of Sciences and by the Hungarian National	Foundation for Scientific Research, Grant No. K124749.

The third author was supported by the Hungarian Academy of Sciences.


\end{document}